\documentclass[11pt]{amsart}
\usepackage{amssymb,amsmath,txfonts,mathrsfs,mathtools,titletoc, tikz, stmaryrd}
\usepackage{color}
\usepackage{graphicx}
\usepackage{float}
\usepackage{appendix}
\usepackage{hyperref}

\usepackage[alphabetic]{amsrefs}
\newtheorem{theorem}{Theorem}[section]
\newtheorem{prop}[theorem]{Proposition}
\newtheorem{lemma}[theorem]{Lemma}
\newtheorem{remark}[theorem]{Remark}

\newtheorem{notation}[theorem]{Notation}
\newtheorem{definition}[theorem]{Definition}

\DeclareMathOperator{\ep}{\epsilon}

\numberwithin{equation}{section}

\def\pf{{\it Proof:}~}

\usepackage{color}
\usepackage{graphicx}
\usepackage{float}
\usepackage{appendix}

\usepackage[alphabetic]{amsrefs}
\numberwithin{equation}{section}

\usepackage[alphabetic]{amsrefs}

\numberwithin{equation}{section}

\def\pf{{\it Proof:}~}

\def\pf{{\it Proof:}~}

\begin{document}

\title[Extrinsic Bonnet-Myers Theorem and almost rigidity]{Extrinsic Bonnet-Myers Theorem and almost rigidity}
\author{Weiying Li, Guoyi Xu}
\address{ Weiying Li\\ School of Mathematical Sciences \\Xiamen University, Xiamen \\P.R. China} 
\address{ Current address: \\ Department of Mathematical Sciences\\Tsinghua University, Beijing\\ P.R. China }
\email{wy-li24@mails.tsinghua.edu.cn}
\address{Guoyi Xu\\Department of Mathematical Sciences\\Tsinghua University, Beijing\\P. R. China, 100084}
\email{guoyixu@tsinghua.edu.cn}
\date{\today}
\date{\today}

\begin{abstract}
We establish the extrinsic Bonnet-Myers Theorem for compact Riemannian manifolds with positive Ricci curvature.  And we show the almost rigidity for compact hypersurfaces, which have positive sectional curvature and almost maximal extrinsic diameter in Euclidean space.  
\\[3mm]
\end{abstract}
\thanks{G. Xu was partially supported by NSFC 12141103.}

\maketitle

\titlecontents{section}[0em]{}{\hspace{.5em}}{}{\titlerule*[1pc]{.}\contentspage}
\titlecontents{subsection}[1.5em]{}{\hspace{.5em}}{}{\titlerule*[1pc]{.}\contentspage}
\tableofcontents
\section{Introduction}

The well-known Bonnet-Myers theorem says: for complete Riemannian manifold $(M^n, g)$ with Ricci curvature $Rc\geq (n- 1)$,  the intrinsic diameter of $(M^n,  g)$ with respect to the Riemannian metric $g$ satisfies $\mathrm{Diam}_g(M^n)\leq \pi$. Furthermore Cheng \cite{Cheng} showed that the rigidity of Bonnet-Myers theorem,  which says that $\mathrm{Diam}_g(M^n)= \pi$ if and only if $(M^n, g)$ is isometric to $\mathbb{S}^n$.  
  
In the rest of this paper,  unless otherwise mentioned,  $(M^n,  g)$ is always a compact Riemannian manifold.

Recall we say that a smooth function $f: (M^n, g)\rightarrow \mathbb{R}^m$ is a smooth isometric embedding, if $f$ is injective and for any coordinate chart $\{x_i\}_{i= 1}^n$ on $M^n$, we have  
\begin{align}
g_{ij}= \langle\frac{\partial f}{\partial x_i},\frac{\partial f}{\partial x_i}\rangle_{\mathbb{R}^m}. \nonumber 
\end{align}
  
We use $\mathcal{IE}((M^n, g),  \mathbb{R}^m)$ to denote the set of all smooth isometric embedding $\mathscr{I}: (M^n,  g)\rightarrow \mathbb{R}^m$,  where $m\geq n+ 1$ and $n\geq 2$ are positive integers.  From the well-known Nash's isometric embedding theorem (see \cite{Nash}),  for any $(M^n, g)$,  there is $m\in \mathbb{Z}^+$ such that $\displaystyle \mathcal{IE}((M^n, g),  \mathbb{R}^m)\neq \emptyset$.

For $\mathscr{I}\in \mathcal{IE}((M^n, g),  \mathbb{R}^m)$,  we define
\begin{align}
\mathrm{Diam}_\mathscr{I}(M^n,  g)\vcentcolon= \sup_{x, y\in M^n}|\mathscr{I}(x)-\mathscr{I}(y)|_{\mathbb{R}^m}.  \nonumber 
\end{align}

If $\mathcal{IE}((M^n, g),  \mathbb{R}^m)\neq \emptyset$,  we define the \textbf{extrinsic diameter of $(M^n, g)$ in $\mathbb{R}^m$} as follows:
\begin{align}
\mathrm{Diam}_{\mathbb{R}^m}(M^n, g)\vcentcolon= \sup_{\mathscr{I}\in \mathcal{IE}((M^n, g),  \mathbb{R}^m)}\mathrm{Diam}_{\mathscr{I}}(M^n, g). \nonumber 
\end{align}

Spruck \cite{Spruck} showed: for any $(M^n, g)$ with sectional curvature $K(g)\geq 1$ and $\mathscr{I}\in \mathcal{IE}((M^n, g),  \mathbb{R}^{n+ 1})$,  there is $\mathrm{Diam}_\mathscr{I}(M^n)< \pi$.  A family of smooth examples was sketched in \cite{Spruck} to show this upper bound is sharp.  Those examples are spheres shrinking to a line segment with length $\pi$ in $\mathbb{R}^{n+ 1}$ (see Proposition \ref{prop max extrin diam examples-smooth in higer dimension} for details).


The first result of this paper is the following extrinsic Bonnet-Myers Theorem, which generalizes the above theorem of Spruck.
\begin{theorem}\label{thm main-0}
{For complete Riemannian manifold $(M^n, g)$ with $Rc\geq (n- 1)$,  we have 
\begin{align}
\mathrm{Diam}_\mathscr{I}(M^n,  g)< \pi, \quad \quad \quad \forall  \mathscr{I}\in \mathcal{IE}((M^n, g),  \mathbb{R}^m). \label{diam wrt f strict}
\end{align}
Furthermore (\ref{diam wrt f strict}) is sharp in the following sense: there exists a sequence of $(S^n, g_k)$ with $K(g_k)\geq 1$ and $\mathscr{I}_k\in\mathcal{IE}((S^n,g_k)$,  $\mathbb{R}^{n+1})$, such that $\displaystyle \lim_{k\rightarrow\infty}\mathrm{Diam}_{\mathscr{I}_k}(S^n, g_k)= \pi$.
}
\end{theorem}

\begin{remark}\label{rem ED in Rm upper bound}
{Although (\ref{diam wrt f strict}) is sharp in the above sense; for complete Riemannian manifold $(M^n, g)$ with $Rc\geq (n- 1)$ and $m\geq n+ 1$, we currently do not know whether $\mathrm{Diam}_{\mathbb{R}^m}(M^n, g)< \pi$ generally holds or not. However, Theorem \ref{thm main-1} provides partial result when $K\geq 1$ and $m= n+ 1$.
}
\end{remark} 

From the rigidity part of Bishop-Gromov's volume comparison Theorem,  for complete Riemannian manifold $(M^n,g)$ with $Rc(g)\ge (n- 1)$,  we have $V(M^n)= V(\mathbb{S}^n)$ if and only if $M^n$ is isometric to $\mathbb{S}^n$. Furthermore, there is almost rigidity with respect to almost maximal volume in the above context. To explain it, we recall some concepts as follows. 

For two subsets $A, B$ of a metric space $Z$, the \textbf{Hausdorff distance} between $A$ and $B$ among $Z$ is 
\begin{align}
d_H^Z (A, B)= \inf\big\{\epsilon> 0: B\subset \mathbf{U}_{\epsilon}(A)\ and \ A\subset \mathbf{U}_{\epsilon}(B)\big\} \nonumber 
\end{align}
where $\displaystyle \mathbf{U}_{\epsilon}(A)\vcentcolon= \big\{z\in Z: d_Z(z,  A)\leq \epsilon\big\}$. The \textbf{Gromov-Hausdorff distance} (also see \cite{Gromov-book}) between two metric space $X, Y$, is denoted as $d_{GH}(X, Y)$,
\begin{align}
d_{GH}(X, Y)= \inf_{\mathscr{I}_1\in \mathscr{IE}(X, Z) \atop \mathscr{I}_2\in \mathscr{IE}(Y, Z)} d_H^Z \big(\mathscr{I}_1(X), \mathscr{I}_2(Y)\big) \nonumber   
\end{align} 
where $Z$ is any metric space with non-empty $\mathscr{IE}(X, Z)$ and $\mathscr{IE}(Y, Z)$; and $\mathscr{IE}(X, Z)$ is the set of all isometric embedding of $X$ into $Z$, similarly $\mathscr{IE}(Y, Z)$ is defined.  If $d_{GH}\big((X, d_X), (Y, d_Y)\big)\leq \epsilon$, we say that $(X, d_X)$ is $\epsilon$-Gromov-Hausdorff close to $(Y, d_Y)$. 

\begin{remark}\label{rem different IE}
{The set $\mathcal{IE}((M^n, g), \mathbb{R}^m)$ contains only smooth isometric embeddings,  which not only keep the distance property but also preserve the property of Riemannian manifolds' curvature.  On the other hand,  the set $\mathscr{IE}(X, Z)$ contains all isometric embeddings between two metric spaces $X, Z$; which is only distance-preserving (comparing \cite{Nash-C1}).  
}
\end{remark}

Colding \cite{Colding-shape},  \cite{Colding-large} (also see \cite{WZ}) proved the almost rigidity of Bishop-Gromov's volume comparison Theorem. More concretely, he showed that $(M^n,  g)$ with $Rc\geq (n- 1)$,  is Gromov-Hausdorff close to $\mathbb{S}^n$,  if and only if the volume of $(M^n, g)$ is almost maximal (i.e.  is close to the volume of $\mathbb{S}^n$).  

Note the model space with respect to almost maximal volume is $\mathbb{S}^n$.  

On the other hand,  there is no almost rigidity for almost maximal intrinsic diameter, although there is Cheng's rigidity theorem of maximal diameter \cite{Cheng}.  There are round spheres with almost maximal intrinsic diameter and `needle type' convex spheres (see examples in Proposition \ref{prop max extrin diam examples-smooth in higer dimension}), both of them have almost intrinsic maximal diameter; but they are not Gromov-Hausdorff close to each other.

However, with respect to the extrinsic diameter, we have the following almost rigidity theorem with the collapsing model $[0, \pi]$.
\begin{theorem}\label{thm main-1}
	For complete Riemannian manifold $(M^n,g)$ with $K(g)\ge1$ and $\displaystyle \mathcal{IE}((M^n, g),  \mathbb{R}^{n+ 1})\neq \emptyset$,  we have
	\begin{align}
	\frac{d_{GH}((M^n, g), [0, \pi])}{\sqrt{\pi- \mathrm{Diam}_{\mathbb{R}^{n+ 1}}(M^n, g)} }  \leq 4\pi^{\frac{3}{2}}  \nonumber .
	\end{align}
\end{theorem}

\begin{remark}\label{rem K replace by Ricci}
{If the assumption $K(g)\geq 1$ is replaced by $Rc(g)\geq n- 1$,  where $n\geq 3$,  we do not know whether the above conclusion is true or not.
}
\end{remark} 

The organization of this paper is as follows. We prove the extrinsic Bonnet-Myers Theorem (Theorem \ref{thm main-0}) in Section \ref{sec extrin BM}. Specially, the examples showing the sharpness of extrinsic diameter upper bound, is constructed in details. And the sharp bound is obtained through applying the Cheng's rigidity Theorem for Bonnet-Myers' Theorem. The assumption of this section is $Rc\geq (n- 1)$, and there is not restriction on the co-dimension of isometric embeddings.

In Section \ref{sec height}, using Toponogov's comparison Theorem, some facts of Euclidean geometry and spherical geometry, we estimate the height of Euclidean triangles in term of the gap between sharp upper bound and extrinsic diameter of manifolds, where the vertexes of those Euclidean triangles are in the image of isometric embedded Riemannian manifolds in Euclidean spaces. Sectional curvature $K(g)\geq 1$ is needed in this section. The results of this section imply that the isometric embedding image of Riemannian manifolds with almost maximal extrinsic diameter lies in an Euclidean neighborhood of a line segment in the ambient Euclidean space. The final main estimate obtained in this section can be viewed as the upper bound of `extrinsic width' of manifolds isometrically embedded into $\mathbb{R}^m$.

Finally, on manifolds with almost maximal extrinsic diameter, we consider `height function', which is the projection map onto the line segment corresponding to the extrinsic diameter. We get the intrinsic diameter's upper bound of the level set of `height function'. This will be obtained by the convexity of isometric embedding image of manifolds and the `extrinsic width estimate'  obtained in Section \ref{sec height}. The convexity relies on the co-dimension of isometric embedding equal to $1$. 
Then we show that the map (which maps the interval to the geodesic segment linking the end points of extrinsic diameter) is a Gromov-Hausdorff approximation, with respect to the scale of the gap between extrinsic diameter and its sharp upper bound. Combining the relationship between Gromov-Hausdorff approximation and Gromov-Hausdorff distance, we get the almost rigidity.

\section{The extrinsic Bonnet-Myers Theorem}\label{sec extrin BM}

We fix some notations, which will be used repeatedly in the rest of the paper.
\begin{notation}\label{notation length geod seg and line}
{For any curve $\gamma\subseteq (M^n, g)$, we use $\ell(\gamma)$ to denote the length of the curve $\gamma$. For $p, q\in (M^n, g)$, we use $\gamma_{p, q}$ to denote one geodesic segment from $p$ to $q$ in $(M^n, g)$. Then $\ell(\gamma_{p, q})= d_g(p, q)$. 

For distinct points $x, y\in \mathbb{R}^m$, we use $l_{x, y}$ to denote the line passing $x, y$ and $\overline{xy}$ to denote the line segment from $x$ to $y$. We use $|x-y|$ or $|x-y|_{\mathbb{R}^m}$ to denote the Euclidean length of $\overline{x y}$.
}
\end{notation}

In this section, for any $\epsilon\in (0, \pi)$, we firstly construct a smooth Riemannian manifold $(S^n,g)\subset\mathbb{R}^{n+1}$ with $K(g)\ge 1$, and some $\mathscr{I}\in\mathcal{IE}((S^n,g),  \mathbb{R}^{n+1})$, such that $\mathrm{Diam}_{\mathscr{I}}(S^n,g)\ge\pi-\epsilon$.  Then we prove the extrinsic Bonnet-Myers Theorem, whose sharpness is guaranteed by the example in Proposition \ref{prop max extrin diam examples-smooth in higer dimension}.

For $(S^n, g)$ with normal coordinate chart $\{t,\theta_1,...,\theta_{n- 1}\}$ and metric $g= dt^2+ f^2(t)d\theta^2$,  where $d\theta=d\theta_1d\theta_2...d\theta_{n-1}$ is the canonical measure on $\mathbb{S}^{n-1}$,  we select an orthonormal basis $\{E_i\}$ where $E_1=\frac{\partial}{\partial t}$,  for all $X,  Y= E_i$ where $i\neq 1$.  The sectional curvatures are as follows:
\begin{align}
	K(E_1,  X)=-\frac{f''}{f}, \quad \quad K(X,  Y)= \frac{1-(f')^2}{f^2},  \quad \quad \forall X\neq Y. \label{sectional curv formula}
\end{align}

The following lemma is well-known (see \cite{PP}). 
\begin{lemma}\label{lem smooth metric criterion}
	If $f:[c, b]\rightarrow [0,\infty)$ is smooth and $f(c)=0$, the metric of $(S^n,  g)$ is $g=dt^2+f^2(t)d\theta^2$,  then $g$ is smooth at $t=c$ if and only if
	\begin{align}
		\nonumber	 f'(c)=1,\quad \quad 	f^{(2k)}(c)=0, \quad \quad \forall k\in \mathbb{Z}^+,
	\end{align}
	where $f^{(2k)}$ is the $2k$-order derivative of $f$. 
\end{lemma}\qed

According to \cite{Spruck}, the example manifold (see Proposition \ref{prop max extrin diam examples-smooth in higer dimension}) was pointed out by Calabi. We provide a detailed construction of the example for completeness reason.

\begin{remark}\label{rem key points of construction}
	{There are two key points of the construction of example manifolds in Proposition \ref{prop max extrin diam examples-smooth in higer dimension}. The first one is to define the twisted factor $f(t)$ as the solution of ODE (\ref{ODE for f}); and this ODE comes from the curvature term $-\frac{f''}{f}\geq 1$. Therefore we reduce the construction of the metric to the choice of suitable function $h$ in (\ref{ODE for f}). 
		
		The second idea is: to solve $f$ with standard initial data at starting point $t= 0$, and get the upper bound of $f'(c)$ where $t= c$ is another end point; then scaling the metric by $|f'(c)|^{-1}$, to guarantee the smoothness of the metric obtained by the new function $\tilde{f}$. 
	}
\end{remark}

\begin{lemma}\label{lem big extrin diam examples with singularity}
	For any $k\geq 100$, there exists $(S^n, g_f)$ with $\displaystyle g_f= dt^2+f^2(t)d\theta^2, t\in [-c, c]$, where $c\geq \frac{\pi}{2}- \frac{1}{k}+ \frac{1}{4k^2}$ and $f$ is an even function; such that $(S^n, g_f)$ is smooth except two points with $t= \pm c$, and 
	\begin{align}
		&K(g_f)\geq 1, \quad \quad \quad \forall t\in [-\frac{\pi}{2}+ \frac{1}{k}, \frac{\pi}{2}- \frac{1}{k}], \nonumber \\
		&f^{(even)}(\pm c)= 0, \quad \quad \quad |f'(\pm c)|\geq  \frac{3k}{16}, \label{singular 1st deri} \\
		&f''(t)+ f(t)\leq 0, \quad \quad  f^2(t)+ (f')^2(t)- (f'(c))^2\leq 0, \quad \quad \quad \forall t\in [-c, c]. \nonumber 
	\end{align}
\end{lemma}

\begin{remark}\label{rem about singularity and negative curv}
	{Because of (\ref{singular 1st deri}) and Lemma \ref{lem smooth metric criterion}, we know that $g_f$ is not a smooth metric at $t= \pm c$. 
		
		From our argument below, the sectional curvature $K(g_f)\geq 0$ does not hold on $(S^n, g_f)$ for all $t\in (-c, c)$, because $1- |f'(t)|^2< 0$ for $t$ is close to $\pm c$.  
	}
\end{remark}

\pf
{\textbf{Step (1)}. We define the smooth function $\varphi(t)\vcentcolon=\frac{\int_{-\infty}^{t}F(s)ds}{\int_{-\infty}^{\infty}F(s)ds}$, where 
	\begin{align}\nonumber 
		F(t)\vcentcolon =\left\{
		\begin{array}{rl}
			e^{\frac{1}{t(t-1)}},0<t<1,\\
			0,\quad \text{otherwise}.
		\end{array} \right.
	\end{align}
	
	Denote $c_0=\frac{1}{4}$ and $\delta= \frac{c_0}{4}= \frac{1}{16}$ in the rest argument. Now we define a smooth function $h:[0, \frac{\pi}{2}]\rightarrow \overline{\mathbb{R}^{-}}$ as follows (see Figure \ref{figure: h}): 
	\begin{equation}\nonumber 
		h(t)= \left\{
		\begin{array}{rl}
			&\varphi(\frac{t- (\frac{\pi}{2}- \frac{1}{k})}{k^{-2}\delta})\cdot \Big(-\delta^{-1}k^{5}\cdot (t- \frac{\pi}{2}+ \frac{1}{k})\Big) \ ,  \quad \quad \quad \quad t\in [0, \frac{\pi}{2}- \frac{1}{k}+ \frac{3}{2}\cdot \frac{\delta}{k^{2}}], \\
			&\Big(1- \varphi(\frac{t- (\frac{\pi}{2}- \frac{1}{k}+ 2\frac{\delta}{k^{2}})}{k^{-2}\delta})\Big)\cdot \Big(-\delta^{-1}k^{5}\cdot (t- \frac{\pi}{2}+ \frac{1}{k})\Big)\ ,   \quad \quad t\in [\frac{\pi}{2}- \frac{1}{k}+ \frac{3}{2}\cdot \frac{\delta}{k^{2}}, \frac{\pi}{2}]. 
		\end{array} \right.
	\end{equation}
	
	\begin{figure}[H]
		\centering
		\includegraphics[height=12cm,width=15cm]{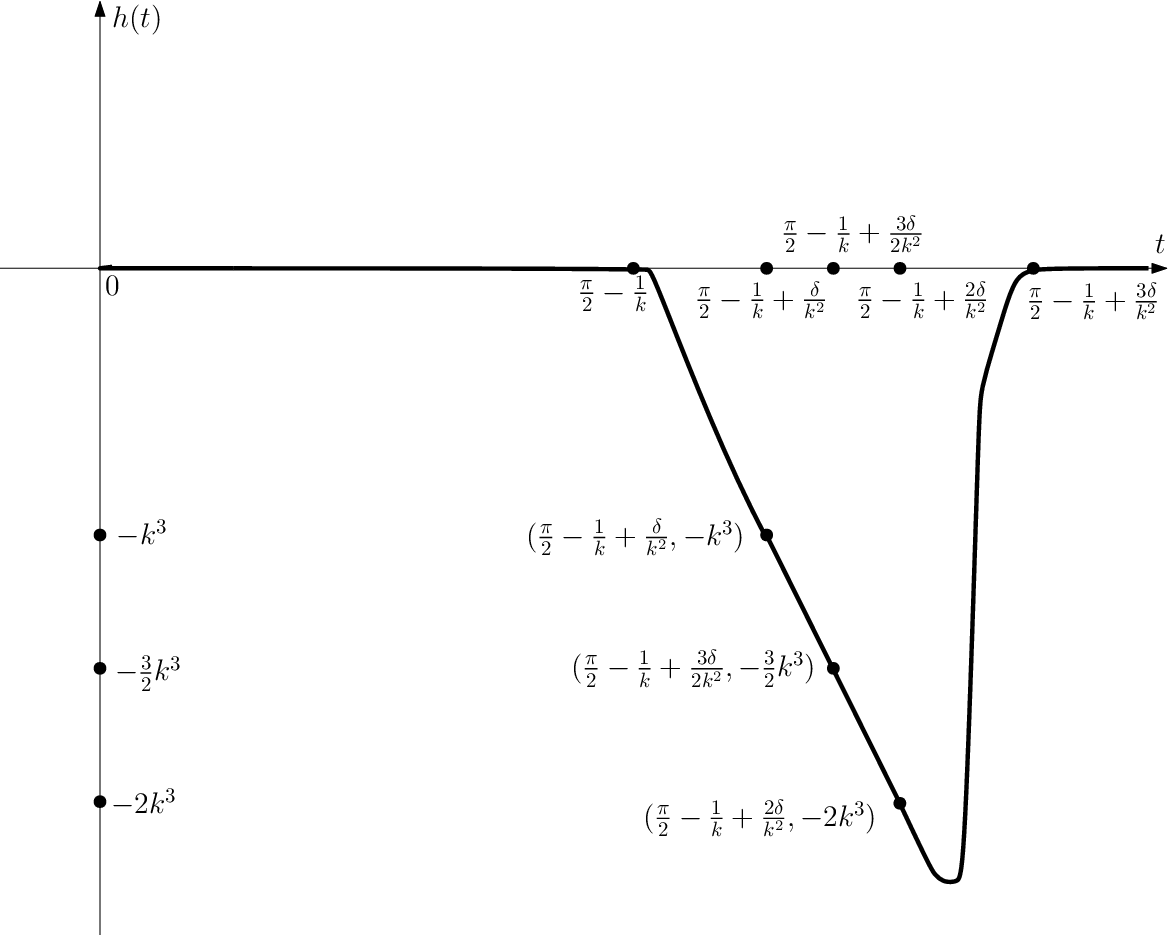}
		\caption{The figure of h(t)}
		\label{figure: h}
	\end{figure}
	
	We assume $f$ is the solution to the following $2$nd order ODE:
	\begin{align}
		f''(t)+ f(t)= h(t), \quad \quad f(0)= 1, \quad \quad f'(0)= 0, \quad \quad \quad t\in [0, \frac{\pi}{2}]. \label{ODE for f}
	\end{align}

Let $c= \min\{t> 0: f(t)= 0\}$, from Bonnet-Myers theorem, $c\in(\frac{\pi}{2}-\frac{1}{k},\frac{\pi}{2})$.

Note when $f(t)> 0$, we always have $f''(t)< 0$. From (\ref{ODE for f}), we know that $f''<0$ in $[0, c]$ and $f'(0)=0$. Then $f'(t)$ is decreasing and less than a negative number when $t\in(0, c)$.

So $f$ decreases to 0 in finite time $[0,c]$. 

Now we have $f(c)= 0$ and $f(t)> 0$ for any $t\in (0, c)$ (see Figure \ref{figure: f}).
	
	\begin{figure}[H]
		\centering
		\includegraphics[height=12cm,width=15cm]{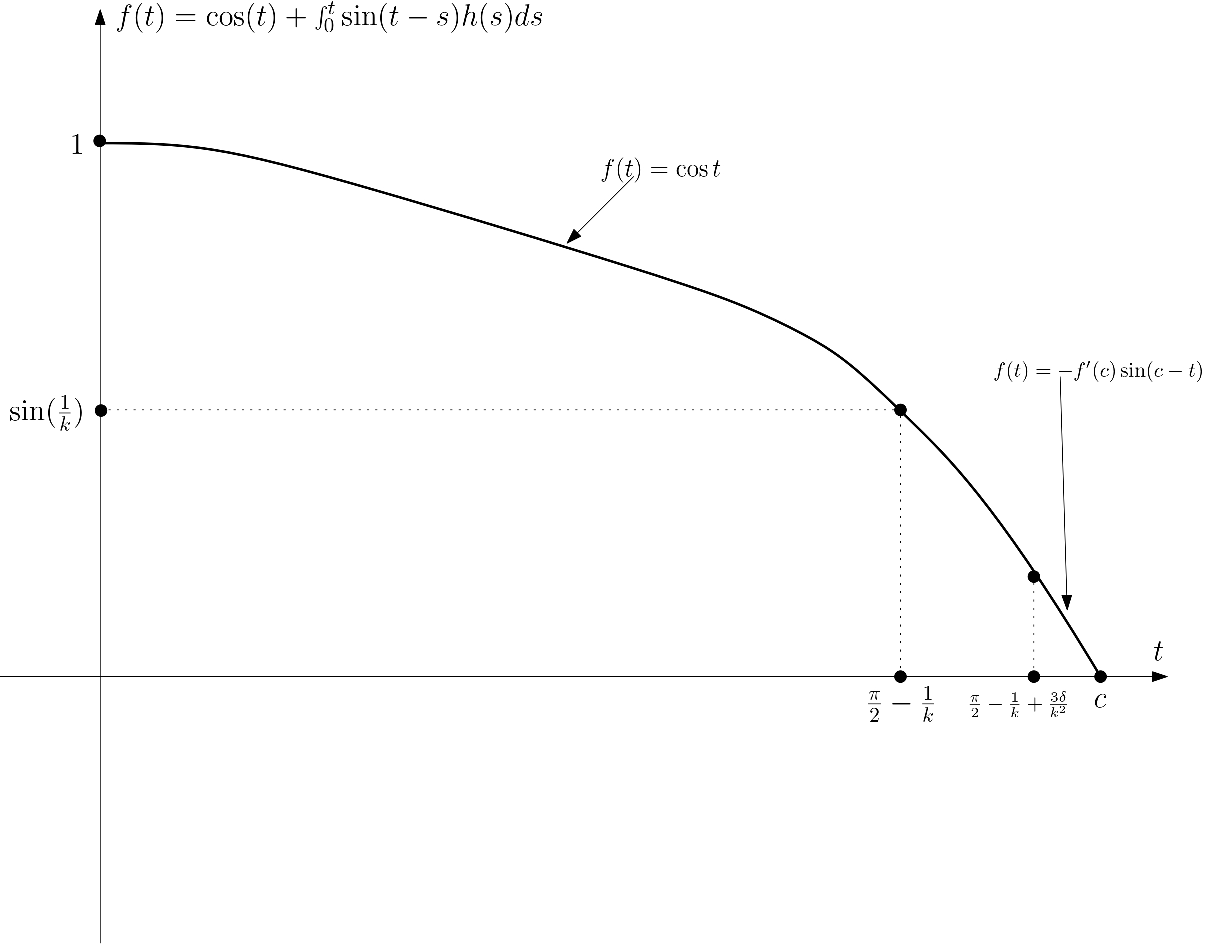}
		\caption{The figure of f(t)}
		\label{figure: f}
	\end{figure}
	
	\textbf{Step (2)}. 
	From the uniqueness of solution to ODE, we know that $f(t)= \cos t$ for any $t\in [0, \frac{\pi}{2}- \frac{1}{k}]$. Now on $[\frac{\pi}{2}- \frac{1}{k}, c)$, we have 
	\begin{align}
		f''(t)= -f(t)+ h(t)\geq -\sin(\frac{1}{k})- 3k^3 \nonumber 
	\end{align}
	
	Then for $t\in [\frac{\pi}{2}- \frac{1}{k}, c)$, we get
	\begin{align}
		f'(t)= f'(\frac{\pi}{2}- \frac{1}{k})+ \int_{\frac{\pi}{2}- \frac{1}{k}}^t f''(s)ds\geq -1- (\sin(\frac{1}{k})+ 3k^3)\cdot (t- [\frac{\pi}{2}- \frac{1}{k}] ). \nonumber 
	\end{align}
	By Newton-Leibniz formula again, let $\tilde{t}= t- [\frac{\pi}{2}- \frac{1}{k}]$, we get
	\begin{align}
		f(t)= f(\frac{\pi}{2}- \frac{1}{k})+ \int_{\frac{\pi}{2}- \frac{1}{k}}^t f'(s)ds\geq \sin(\frac{1}{k})- \tilde{t}- \frac{\tilde{t}^2}{2}(\sin(\frac{1}{k})+ 3k^3). \nonumber 
	\end{align}
	If $\displaystyle 0\le\tilde{t}<\frac{-1+\sqrt{1+2\sin(\frac{1}{k})(\sin(\frac{1}{k})+ 3k^3)}}{\sin(\frac{1}{k})+ 3k^3}$, we have  $f(t)>0$. 
	
	For any $k\geq 1$, using $\frac{\sin x}{x}\in [\frac{2}{\pi}, 2]$ where $x\in (0, \frac{\pi}{2})$, we obtain 
	\begin{align}
		&\frac{-1+\sqrt{1+2\sin(\frac{1}{k})(\sin(\frac{1}{k})+ 3k^3)}}{\sin(\frac{1}{k})+ 3k^3}k^2\geq \frac{-1+\sqrt{1+2\cdot \frac{2}{\pi}\cdot \frac{1}{k}\cdot  3k^3}}{1+ 3k^3}k^2 \nonumber \\
		&\geq \frac{-1+ 2k}{1+ 3k}\geq \frac{1}{4}= c_0. \nonumber 
	\end{align}
	
	Therefore $f(t)> 0$ if $\tilde{t}\leq c_0k^{-2}$, and
	\begin{align}
		c\geq (\frac{\pi}{2}- \frac{1}{k})+ c_0k^{-2} =\frac{\pi}{2}- \frac{1}{k}+ \frac{4\delta}{k^2}. \label{lower bound of c intrinsic diam}
	\end{align}
	
	By the definition of $h(t),$ we have $h^{(m)}(c)=0$, for any $m\in\mathbb{N}.$ So $f''(c)+f(c)=h(c)=0.$ Since $f(c)=0$, then $f''(c)=0.$ By $f^{(m+2)}(c)+f^{(m)}(c)=h^{(m)}(c)$, we have $f^{(even)}(c)=0.$
	
	For any $t\in [0, c]$, using Newton-Leibniz formula and $f(c)= 0, f''(s)= h(s)- f(s), f'(s)\leq 0$, we have
	\begin{align}
		& f^2(t)+ (f')^2(t)- (f'(c))^2= f^2(t)- f^2(c)- \int_t^c 2f'(s)f''(s)ds \nonumber \\
		&=\int_{c}^{t}2f(s)f'(s)+2f'(s)f''(s)ds=-\int_{t}^{c}2f'(s)h(s)ds\le 0. \label{curv ineq for higher dim}
	\end{align}
	
	By the decreasing property of $f'$ on $(0, c)$, we get 
	\begin{align}
		f'(c)&\leq f'((\frac{\pi}{2}- \frac{1}{k})+ c_0k^{-2})= f'(\frac{\pi}{2}- \frac{1}{k})+ \int_{\frac{\pi}{2}- \frac{1}{k}}^{(\frac{\pi}{2}- \frac{1}{k})+ c_0k^{-2}} f''(s)ds \nonumber \\
		&\leq \int_{\frac{\pi}{2}- \frac{1}{k}+ \delta\cdot k^{-2}}^{(\frac{\pi}{2}- \frac{1}{k})+ 4\delta k^{-2}} -k^3ds= -3\delta k. \label{crucial bound of f'(c)}
	\end{align}
}
\qed

\begin{remark}\label{rem round off singularity}
	{To ``round off" those two singularities $t= \pm c$, we only need to scale metric on $d\theta^2$ factor by $\frac{1}{|f'(c)|}= \frac{1}{|f'(-c)|}$; which is done in the argument of Proposition \ref{prop max extrin diam examples-smooth in higer dimension}. 
	}
\end{remark}

\begin{lemma}\label{lem hypersurface and wraped metric}
	{For a smooth function $f: [-c, c]\rightarrow \mathbb{R}$ with 
\begin{align}
f\big|_{(-c, c)}> 0, \quad \quad |f'|\leq 1, \quad \quad  f^{(even)}(\pm c)= 0, \quad \quad f'(\pm 1)= 0;\nonumber 
\end{align}	
define a Riemannian manifold $(S^n, g_f)$ with $g_f= dt^2+ f^2(t)d\theta^2$, where $t\in [-c, c]$, $d\theta^2$ is the canonical metric of $\mathbb{S}^{n-1}$, and $\theta_1, \cdots, \theta_{n- 2}\in [0, \pi], \theta_{n- 1}\in [0, 2\pi]$ is the coordinate system of $\mathbb{S}^{n- 1}$. Then there is $\mathscr{I}_f\in \mathcal{IE}((S^n,g),  \mathbb{R}^{n+1})$, where 
	\begin{align}
	& \mathscr{I}_f(t, \theta_1, \cdots, \theta_{n- 1}) \nonumber \\
&	= (f(t)\cos\theta_1,f(t)\sin\theta_1\cos\theta_2,..., f(t)\cos \theta_{n- 1}\prod_{i= 1}^{n- 2}\sin\theta_i, \nonumber \\
&\quad \quad f(t)\sin \theta_{n- 1}\prod_{i= 1}^{n- 2}\sin\theta_i,\int_{0}^{t}\sqrt{1-f^{\prime2}(s)}ds). \nonumber 
	\end{align}
	}
\end{lemma}

\pf
{It is trivial.
}
\qed	

\begin{remark}\label{rem not always positive curved}
	{Some part of the Riemannian manifold $(S^n, g_f)$ (from Lemma \ref{lem big extrin diam examples with singularity}) with $|f'|> 1$,  can not be isometrically embedded into $\mathbb{R}^{n+ 1}$ by the isometric embedding $\mathscr{I}$ in Lemma \ref{lem hypersurface and wraped metric}.
	}
\end{remark}

\begin{prop}\label{prop max extrin diam examples-smooth in higer dimension}
	For any $\epsilon> 0,$ there exists a smooth hypersurface $(S^n, g)$ with $K(g)\geq 1$ with $\mathscr{I}\in\mathcal{IE}((S^n,g),  \mathbb{R}^{n+1})$, such that $\mathrm{Diam}_{\mathscr{I}}(S^n,g)\ge\pi-\epsilon$ . 
\end{prop}

\pf
{For $k\geq 100$ to be determined later, choose $f$ from Lemma \ref{lem big extrin diam examples with singularity}, let $\tilde{f}(t)=\frac{f(t)}{|f'(c)|}: [-c, c]\rightarrow \mathbb{R}$, then 
	\begin{align}
		&\tilde{f}''+ \tilde{f}\leq 0,\quad \tilde{f}^2+(\tilde{f}')^{2}-1\le0, \label{curv geq 1 assumption} \\
		&	\tilde{f}(0)= \frac{1}{|f'(c)|},  \quad \tilde{f}'(0)= 0,\quad 	\tilde{f}(c)= 0,\quad \tilde{f}'(c)= -1,\quad	\tilde{f}^{(even)}(c)=0.  \label{f satisfy deri assump}
	\end{align}
	
	Consider $(S^n, g_{\tilde{f}})\subseteq \mathbb{R}^{n+ 1}$ defined by $\tilde{f}$ as in Lemma \ref{lem hypersurface and wraped metric}, where $g_{\tilde{f}}= dt^2+ \tilde{f}^2(t)d\theta^2$. Then $g_{\tilde{f}}$ is smooth by (\ref{f satisfy deri assump}) and Lemma \ref{lem smooth metric criterion} (see Figure \ref{figure: scaling surfaces}, where $|f'(\pm t_0)|=1$).
	
	\begin{figure}[H]
		\centering
		\includegraphics[width=15cm,height=15cm]{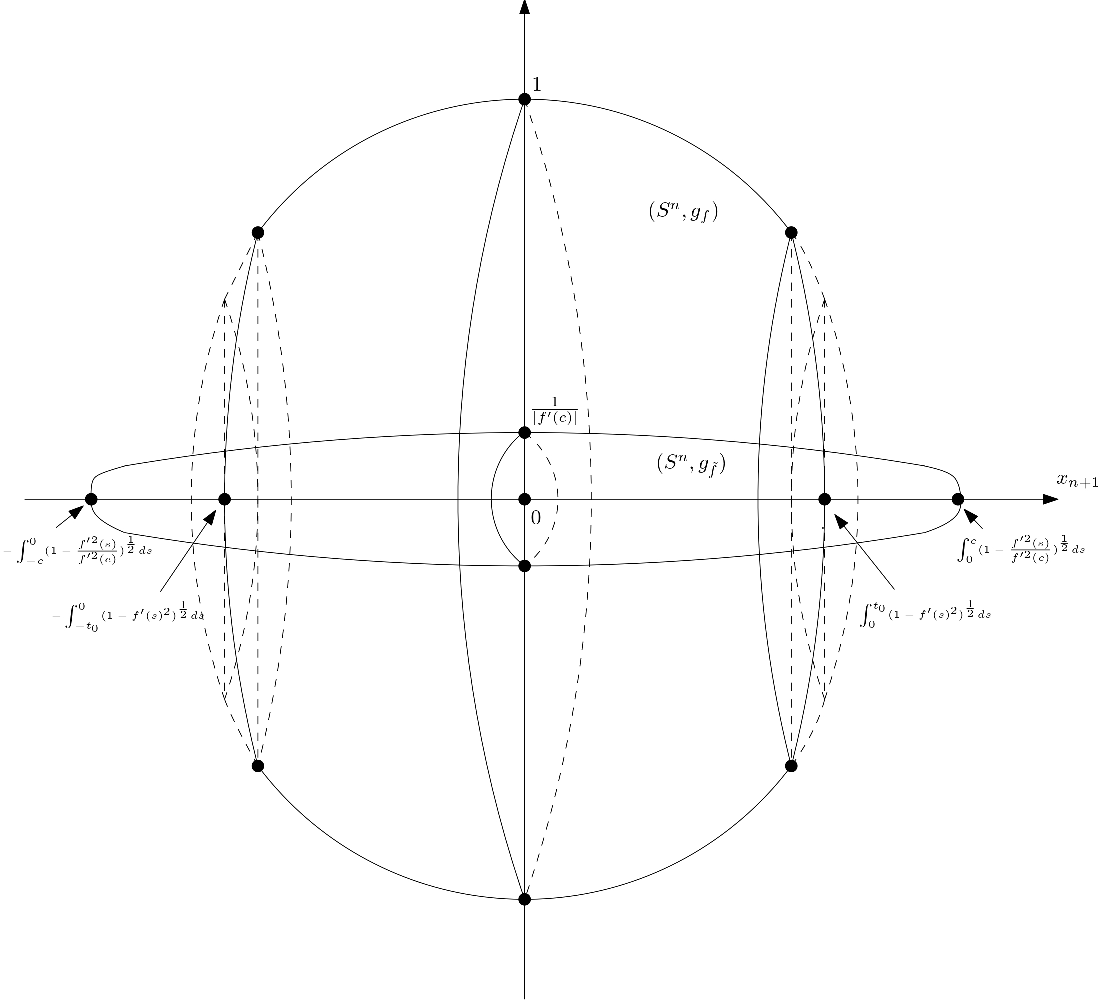}
		\caption{Scaling the slice spheres to smooth the two ends singularities}
		\label{figure: scaling surfaces}
	\end{figure}
	
	From (\ref{curv geq 1 assumption}) and (\ref{sectional curv formula}), we get that $K(g_{\tilde{f}})\ge1$. From $\tilde{f}''\leq 0, \tilde{f}'(0)= 0$, there is $-1\leq \tilde{f}'\big|_{[0, c]}\leq 0$. Using (\ref{lower bound of c intrinsic diam}) and (\ref{crucial bound of f'(c)}), we have
	\begin{align}
		\mathrm{Diam}_{\mathscr{I}_{\tilde{f}}}(S^n, g_{\tilde{f}})&\geq 2\int_{0}^{c}\sqrt{1-\tilde{f}'^2(s)}ds\geq 2\int_{0}^{c}1+\tilde{f}'(s)ds=2c- 2\tilde{f}(0)
		\nonumber \\
		&	\geq \pi- \frac{2}{k}- \frac{2}{|f'(c)|}\geq \pi- \frac{20}{k}. \nonumber 
	\end{align}
	If $k$ is big enough, we get that $\mathrm{Diam}_{\mathscr{I}_{\tilde{f}}}(S^n, g_{\tilde{f}})> \pi- \epsilon$, the conclusion is proved.
}
\qed

\begin{theorem}\label{thm strict ineq for fixed IE}
	{For complete Riemannian manifold $(M^n, g)$ with $Rc\geq (n- 1)$,  we have 
		\begin{align}
			\mathrm{Diam}_\mathscr{I}(M^n,  g)< \pi, \quad \quad \quad \forall  \mathscr{I}\in \mathcal{IE}((M^n, g),  \mathbb{R}^m). \label{strict ineq of diam}
		\end{align}
		Furthermore (\ref{strict ineq of diam}) is sharp in the following sense: there exists a sequence of $(S^n, g_k)$ with $K(g_k)\geq 1$ with $\mathscr{I}_k\in\mathcal{IE}((S^n,g_k),  \mathbb{R}^{n+1})$ and $\displaystyle \lim_{k\rightarrow\infty}\mathrm{Diam}_{\mathscr{I}_k}(S^n, g_k)= \pi$.
		}
\end{theorem}

\pf
{We prove (\ref{strict ineq of diam}) by contradiction, assume $\mathrm{Diam}_\mathscr{I}(M^n,  g)=\pi$ for some $\mathscr{I}\in \mathcal{IE}((M^n, g),  \mathbb{R}^m)$.  
	
	From Bonnet-Myers' Theorem and $Rc(g)\geq n-1$,  we have
	\begin{align}
		\mathrm{Diam}_{g}(M^n)\leq \pi. \nonumber 
	\end{align}
	Therefore $\mathrm{Diam}_\mathscr{I}(M^n,  g)\leq \mathrm{Diam}_{g}(M^n)\leq \pi$. 
	
By $\mathrm{Diam}_\mathscr{I}(M^n,  g)=\pi$,  we have $\mathrm{Diam}_{g}(M^n)=\pi.$ Assume for $p, q\in M^n$,  we have 
	\begin{align}
	|\mathscr{I}(p)- \mathscr{I}(q)|_{\mathbb{R}^m}= \pi.  \nonumber 
	\end{align}	
	
	Denote the unit speed geodesic segment from $\mathscr{I}(p)$ to $\mathscr{I}(q)$ in $(\mathscr{I}(M^n),  g)$ as $\gamma_{\mathscr{I}(p), \mathscr{I}(q)}$. If $\gamma_{\mathscr{I}(p), \mathscr{I}(q)}$ is not a line segment in $\mathbb{R}^m$,  then 
	\begin{align}
		d_g(p, q)= d_g(\mathscr{I}(p), \mathscr{I}(q))> |\mathscr{I}(p)- \mathscr{I}(q)|_{\mathbb{R}^m}= \pi. \nonumber 
	\end{align}
	This contradicts $\mathrm{Diam}_{g}(M^n)=\pi.$ So $\gamma_{\mathscr{I}(p), \mathscr{I}(q)}$ is a line segment in $\mathbb{R}^m$. 
	
	 Assume $\tilde{\nabla}$ is the Levi-Civita connection of $\mathbb{R}^{m}$, $\nabla$ is the Levi-Civita connection of $(M^n,  g)$. For simplicity, we use $\gamma$ to denote $\gamma_{\mathscr{I}(p), \mathscr{I}(q)}$ in the rest argument. 
	 
	 Choose the parallel unit orthogonal frame $\{e_i\}_{i=1}^n$ along $\gamma(t)$ with $e_1=\gamma'(t)$. By Gauss equation, for $2\leq i\leq n$, we have
	 \begin{align}
	 \widetilde{Rm}(e_1,e_{i},e_1,e_{i})=Rm(e_1,e_{i},e_1,e_{i})-\langle S(e_1,e_1),S(e_i,e_i) \rangle+|S(e_1,e_i)|^2,\label{Gauss formula}
	 \end{align}
where $\widetilde{Rm}$ is the Riemannian curvature of $\mathbb{R}^{m}$, $Rm$ is the Riemannian curvature of $(M^n,  g),$ $S(X,Y)=-(\tilde{\nabla}_XY)^{\perp}$ is the second fundamental form of $(M^n, g)\subseteq \mathbb{R}^m$.
	
	Since $\gamma(t)$ is a line segment in $\mathbb{R}^m$, we have $\tilde{\nabla}_{e_1}e_1=0$, then $S(e_1,e_1)=0.$ Now from (\ref{Gauss formula}) and $\widetilde{Rm}= 0$, we get
\begin{align}
Rc(e_1,e_1)=\sum\limits_{i= 2}^nRm(e_1,e_{i},e_1,e_{i})=\sum\limits_{i= 2}^n-|S(e_1,e_i)|^2\le0
\end{align}
	which contradicts $Rc(g)\ge(n-1).$ 

    Hence $\mathrm{Diam}_{\mathscr{I}}(M^n, g)<\pi.$ The sharpness of (\ref{strict ineq of diam}) follows from Proposition \ref{prop max extrin diam examples-smooth in higer dimension}. 
}
\qed

\section{The height of triangles}\label{sec height}

A geodesic \textbf{hinge} $(c_1, c_0, \alpha)$ in $(M^n, g)$ consists of two nonconstant geodesic segments $c_1, c_0$ with the same initial point making the angle $\alpha$. A geodesic segment $c$ between the endpoints of $c_1$ and $c_0$ is called a \textbf{closing edge} of the hinge. We recall the Toponogov's theorem (see \cite{CE}) as follows.
\begin{theorem}[Toponogov]\label{thm Toponogov}
Let $(M^n, g)$ be a complete Riemannian manifold with $K(g)\geq 1$. Let $(c_0, c_1, \alpha)$ be a hinge in $M$ and $c$ a closing edge. Then the closing edge $\tilde{c}$ of any hinge $(\tilde{c}_0, \tilde{c}_1, \alpha)$ in $\mathbb{S}^n$ with $\ell(\tilde{c}_i)= \ell(c_i), i= 0, 1$, satisfies $\ell(\tilde{c})\geq \ell(c)$.
\end{theorem}\qed

Recall the excess function defined in \cite{AG} as follows. 
\begin{definition}\label{def excess func}
	{For $p, q\in (M^n, g)$, we define the \textbf{excess function with respect to $p, q$}, denoted as $\mathbf{E}_{p, q}: M^n\rightarrow \mathbb{R}$, as:
		\begin{align}
			\mathbf{E}_{p, q}(x)= d_g(p, x)+ d_g(q, x)- d_g(p, q). \nonumber 
		\end{align}
	}
\end{definition}

\begin{lemma}\label{lem excess est for K geq 1}
	For complete Riemannian manifold $(M^n,g)$ with $K(g)\ge 1$, we have 
	\begin{align}
		\sup_{x\in M^n}\mathbf{E}_{p, q}(x)\leq 2(\pi- d_g(p, q)), \quad \quad \quad \forall p, q\in M^n. \nonumber 
	\end{align}
\end{lemma}

\pf
{Applying Toponogov Theorem on hinges $\{\gamma_{x, p}, \gamma_{x, q}, \alpha\}\subseteq (M^n, g)$, we can find a spherical triangle $\{\gamma_{\tilde{x}, \tilde{p}}, \gamma_{\tilde{x}, \tilde{q}}, \alpha\}\subseteq \mathbb{S}^n$, where $\alpha$ are angles between geodesic segments $\gamma_{x, p}, \gamma_{x, q}$ and
	\begin{align}
		d_g(p, x)= d_{\mathbb{S}^n}(\tilde{p}, \tilde{x}), \quad \quad d_g(q, x)= d_{\mathbb{S}^n}(\tilde{q}, \tilde{x}), \quad \quad d_{\mathbb{S}^n}(\tilde{p}, \tilde{q})\geq d_g(p, q). \nonumber 
	\end{align}
	
	From spherical geometry, we know that
	\begin{align}
		d_{\mathbb{S}^n}(\tilde{p}, \tilde{x})+ d_{\mathbb{S}^n}(\tilde{q}, \tilde{x})\leq 2\pi- d_{\mathbb{S}^n}(\tilde{p}, \tilde{q}). \nonumber 
	\end{align}
	
	Using $d_{\mathbb{S}^n}(\tilde{p}, \tilde{q})\geq d_g(p, q)$, then 
	\begin{align}
		&\mathbf{E}_{p, q}(x)= d_g(p, x)+ d_g(q, x)- d_g(p, q)= d_{\mathbb{S}^n}(\tilde{p}, \tilde{x})+ d_{\mathbb{S}^n}(\tilde{q}, \tilde{x})- d_g(p, q) \nonumber \\
		&\leq 2\pi- d_{\mathbb{S}^n}(\tilde{p}, \tilde{q})- d_g(p, q)\leq 2\pi- 2d_g(p, q).\nonumber 
	\end{align}
}
\qed

\begin{lemma}\label{lem height by excess}
	{For any triangle $\triangle abc\subseteq \mathbb{R}^2$, we have 
		\begin{align}
			d(a,  l_{b, c})^2\le\frac{(|b-a|+|c- a|)^2-|b-c|^2}{4} . \nonumber 
		\end{align}
	}
\end{lemma}

\pf
{Assume $\overline{aa_0}\perp l_{b, c}$, where $a_0\in l_{b, c}$ (note $a_0$ possibly does not belong to $\overline{bc}$).  Then $d(a,  l_{b, c})= |a- a_0|$.  Now from the Cosine Law 
	\begin{align}
		d(a,  l_{b, c})^2&=\frac{1}{|b-c|^2}(\frac{(|b-a|+|c- a|)^2-|b-c|^2}{2})	\cdot(\frac{|b-c|^2-(|b-a|-|c- a|)^2}{2}) \nonumber \\
		\nonumber &\le\frac{(|b-a|+|c- a|)^2-|b-c|^2}{4} . \nonumber 
	\end{align}
}
\qed

Now we show the height estimate of an Euclidean triangle,  whose vertexes are in the image of $\mathscr{I}(M^n)$,  where $(M^n, g)$ has $K(g)\geq 1$ and $\mathscr{I}\in \mathcal{IE}((M^n, g), \mathbb{R}^m)$. 

\begin{prop}\label{prop one-side GH-appr}
	For complete Riemannian manifold $(M^n,g)$ with $K(g)\ge 1$ and $m\geq n+ 1$, assume $p, q\in M^n$. Then 
	\begin{align}
		\sup_{a\in M^n} \frac{d_{\mathbb{R}^m}(\mathscr{I}(a), l_{\mathscr{I}(p), \mathscr{I}(q)})}{\sqrt{\pi- |\mathscr{I}(p)- \mathscr{I}(q)|_{\mathbb{R}^m}}}\leq \sqrt{\pi},  \quad \quad \quad \forall \mathscr{I}\in \mathcal{IE}((M^n, g),  \mathbb{R}^m).\label{extrin estimate}
	\end{align}
\end{prop}

\begin{remark}\label{rem extrin width}
	{If $\mathrm{Diam}_\mathscr{I}(M^n, g)= \pi- \epsilon$ for some positive $\epsilon<< 1$, choose $p, q$ such that $|\mathscr{I}(p)- \mathscr{I}(q)|_{\mathbb{R}^m}= \pi- \epsilon$; then Proposition \ref{prop one-side GH-appr} gives the ``extrinsic width" estimates for $\mathscr{I}(M^n)\subseteq \mathbb{R}^m$.  
	}
\end{remark}

\pf
{Assume $|\mathscr{I}(p)-\mathscr{I}(q)|_{\mathbb{R}^m}=\pi-\epsilon$, where $\epsilon> 0$. Assume $d_{g}(p,q)=\pi-\delta$,  then $\delta\leq \epsilon$. By Lemma \ref{lem excess est for K geq 1}, we get
	\begin{align}
		\nonumber \sup_{x\in M^n}\mathbf{E}_{p, q}(x)+d_g(p,q)\le2\pi-d_g(p,q)=\pi+\delta\le\pi+\epsilon.
	\end{align}
	
	Consider the Euclidean triangle $\Delta \mathscr{I}(p)\mathscr{I}(a)\mathscr{I}(q)\subseteq \mathbb{R}^{m}$ with $\overline{\mathscr{I}(a)b}\perp l_{\mathscr{I}(p), \mathscr{I}(q)}$, where $b\in l_{\mathscr{I}(p), \mathscr{I}(q)}$. 	We define $h\vcentcolon= |\mathscr{I}(a)- b|$ (see Figure \ref{figure: E-Triangle}), then from Lemma \ref{lem height by excess} and Lemma \ref{lem excess est for K geq 1}, we obtain
	\begin{align}
		h^2 &\le\frac{(|\mathscr{I}(p)-\mathscr{I}(a)|+|\mathscr{I}(q)- \mathscr{I}(a)|)^2-|\mathscr{I}(p)-\mathscr{I}(q)|^2}{4} \nonumber \\
		&		\le\frac{(d_g(p,a)+d_g(q,a))^2-|\mathscr{I}(p)-\mathscr{I}(q)|^2}{4}  \nonumber \\
		&=(\mathbf{E}_{p, q}(a)+ d_g(p,q)- |\mathscr{I}(p)- \mathscr{I}(q)|)\cdot \frac{d_g(p,a)+d_g(q,a)+ |\mathscr{I}(p)- \mathscr{I}(q)|}{4}\nonumber \\
		\nonumber& \le\frac{(\pi+\epsilon-(\pi-\epsilon))(\mathbf{E}_{p, q}(a)+ 2d_g(p, q))}{4}\leq\epsilon\cdot \pi. 
	\end{align}
	The conclusion follows.
	\begin{figure}[H]
		\centering
		\includegraphics[width=7cm,height=10cm]{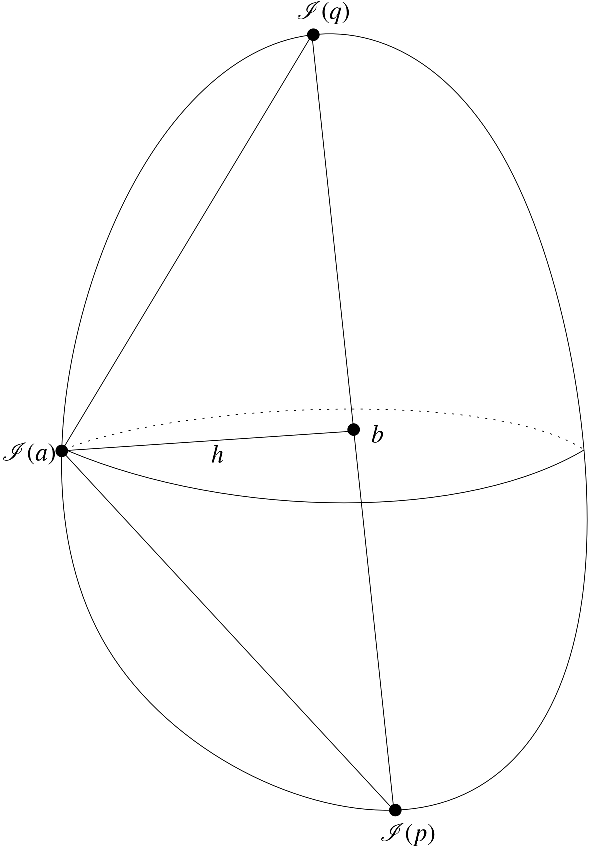}
		\caption{The Euclidean Triangle}
		\label{figure: E-Triangle}
	\end{figure}
}
\qed

\section{The almost rigidity for extrinsic diameter}\label{sec almost rigidity}

From Remark \ref{rem extrin width}, we know that $\mathscr{I}(M^n)$ is in a small neighborhood of the line segment $\overline{\mathscr{I}(p)\mathscr{I}(q)}\subseteq \mathbb{R}^m$, where $\mathscr{I}\in \mathcal{IE}((M^n, g), \mathbb{R}^m)$ and $|\mathscr{I}(p)- \mathscr{I}(q)|$ is the extrinsic diameter of $\mathscr{I}(M^n)$ in $\mathbb{R}^m$. In other words, the manifold $\mathscr{I}(M^n)\subseteq \mathbb{R}^m$ is close to $\overline{\mathscr{I}(p)\mathscr{I}(q)}$ in $\mathbb{R}^m$.

However, to get the upper bound of the Gromov-Hausdorff distance between $\mathscr{I}(M^n)$ and $\overline{\mathscr{I}(p)\mathscr{I}(q)}$, we also need suitable information about the second fundamental form $\mathscr{I}(M^n)\subseteq \mathbb{R}^m$. 

When the co-dimension of $\mathscr{I}(M^n)$ is $1$, we get the positiveness of the second fundamental form for $\mathscr{I}(M^n)\subseteq \mathbb{R}^{n+ 1}$ as follows.
\begin{lemma}\label{lem from Rc>0 to prin curv > 0}
If $(M^n,g)\subset\mathbb{R}^{n+1}$ is a compact Riemannian manifold with $Rc(g)> 0$,  then $(M^n,  g)$ is a closed,  strictly convex hypersurface in $\mathbb{R}^{n+ 1}$.
\end{lemma}

\begin{proof}
Without loss of generality,  we assume $\{\lambda_i\}_{i=1}^n$ are the principal curvatures of $(M^n,g)\subset\mathbb{R}^{n+1}$ and $\lambda_1\leq \cdots \leq \lambda_n$.  We firstly show that  the principal curvature $\lambda_i(M^n)> 0$ for $i= 1, \cdots, n$.

If $\lambda_n< 0$ or $\lambda_1> 0$,  we are done.  

Otherwise, there is $1\leq i_0\leq n$ such that
\begin{align}
\lambda_{i_0}\leq 0\leq \lambda_{i_0+ 1}.  \label{contradiction impliciation} 
\end{align}

From the Gauss equation on $(M^n, g)\subseteq \mathbb{R}^{n+ 1}$, the Ricci curvature of $g$ is as follows:
 \begin{align}
& R_{i_0i_0}=\lambda_{i_0}(H- \lambda_{i_0})>0,  \label{Rc i} \\
 &R_{i_0+ 1, i_0+ 1}= \lambda_{i_0+ 1}(H- \lambda_{i_0+ 1})>0,  \label{Rc i+ 1}
 \end{align}
	where $\displaystyle H= \sum_{i= 1}^n \lambda_i$ is the mean curvature. 
	
	From (\ref{contradiction impliciation}) and (\ref{Rc i}),  we get 
	\begin{align}
	H- \lambda_{i_0}< 0.  \label{1st ineq for i}
	\end{align}
	
	By (\ref{contradiction impliciation}) and (\ref{Rc i+ 1}),  we have
	\begin{align}
	H- \lambda_{i_0+ 1}> 0.  \label{2nd ineq for i+ 1}
	\end{align}
	
Now by (\ref{1st ineq for i}) and (\ref{2nd ineq for i+ 1}),  we obtain $\displaystyle \lambda_{i_0+ 1}< H< \lambda_{i_0}$.   It is the contradiction.  
	
The conclusion follows from that all $\lambda_i> 0$ and \cite[Theorem,  page $241$]{VH}.  
\end{proof}

The following estimate for convex hypersurface is used to control the Gromov-Hausdorff distance in Theorem \ref{thm almost rigidity}. 
\begin{lemma}\label{lem length comp for convex hypersurface}
{For $r> 0$ and $n\geq 1$, if $\Sigma^{n}\subseteq B(r)\subseteq \mathbb{R}^{n+ 1}$ is a closed convex hypersurface, then 
\begin{align}
\mathcal{H}^{n}(\Sigma^{n})\leq \mathcal{H}^{n}(\partial B(r)), \nonumber 
\end{align}
where $\mathcal{H}^n$ is $n$-dimensional Hausdorff measure and $B(r)$ is the ball with radius $r$ in $\mathbb{R}^{n+ 1}$.
}
\end{lemma}

\pf
{
	Let $\Omega$ be the convex set enclosed by $\Sigma^n$ with $\partial\Omega= \Sigma^n$. Define the map	$\mathcal{P}: \overline{B(r)}\rightarrow \Omega$ as
	\begin{align}
	d(x, \mathcal{P}(x))= \inf_{y\in \Omega}d(x, y), \nonumber 
	\end{align}
	 which is a well-defined Lipschitz map with Lipschitz constant $\leq 1$(see \cite[Theorem $5.2$ and Proposition $5.3$]{Brezis}).

It is easy to get that $\mathcal{P}(\partial B(r))\subseteq \partial\Omega= \Sigma^n$. Now $\mathcal{H}^{n}(\Sigma^{n})\le\mathcal{H}^{n}(\partial B(r))$ follows from the area formula for Lipschitz map $\mathcal{P}$ (see \cite{EG}).
}
\qed

Let $(X, d_X)$ and $(Y, d_Y)$ be two metric spaces, a map $F: X\rightarrow Y$ is called an \textbf{$\epsilon$-Gromov-Hausdorff approximation} if 
\begin{align}
Y\subset \mathbf{U}_{\epsilon}\Big(F\big(X\big)\Big), \quad \quad \sup_{x_1, x_2\in X}\Big|d_Y\big(F(x_1), F(x_2)\big)- d_X(x_1, x_2)\Big|\leq \epsilon.  \nonumber 
\end{align}

The following lemma is closely related to \cite[$3.4 (d_+)$,  Proposition $3.5$]{Gromov-book}.  
\begin{lemma}\label{lem GH map implies GH dist control}
{Let $(X, d_X)$ and $(Y, d_Y)$ be two metric spaces,  if there is an $\epsilon$-Gromov-Hausdorff approximation $F: X\rightarrow Y$,  then $\displaystyle d_{GH}(X, Y)\leq 4\epsilon$.
}
\end{lemma}

\pf
{\textbf{Step (1)}.  We choose an $\epsilon$-dense net $\{x_i\}_{i\in I}$ of $X$,  define $y_i= F(x_i)\in Y$.  Let $Z= X\sqcup Y$,  define $d_Z\big|_{X}= d_X,  d_Z\big|_{Y}= d_Y$ and 
\begin{align}
d_Z(x, y)= \epsilon+ \inf_{i}[d_X(x, x_i)+ d_Y(y, y_i)],  \quad \quad \forall x\in X, y\in Y.  \nonumber 
\end{align}

We can verify that $(Z, d_Z)$ is a metric space and $X,  Y$ are isometrically embedded into $Z$.  

\textbf{Step (2)}.  Note for any $y\in Y$,  because $F$ is an $\epsilon$-Gromov-Hausdorff approximation from $X$ to $Y$,  there is $x\in X$ such that 
\begin{align}
d_Y(y, F(x))\leq \epsilon.  \nonumber 
\end{align}

Since $\{x_i\}$ is an $\epsilon$-dense net in $X$,  there is $i_0\in I$ such that $d_X(x,  x_{i_0})\leq \epsilon$.  Then 
\begin{align}
d_Z(x_{i_0}, y)&\leq \epsilon+ d_X(x_{i_0}, x_{i_0})+ d_Y(y, y_{i_0}) = \epsilon+  d_Y(y,  F(x_{i_0}))\nonumber \\
&\leq \epsilon+ d_Y(y,  F(x))+ d_Y(F(x), F(x_{i_0}))\leq 2\epsilon+ d_X(x, x_{i_0})+ \epsilon\leq 4\epsilon.  \nonumber 
\end{align}

From the above,  we obtain $Y\subseteq \mathbf{U}_{4\epsilon}(X)\subseteq Z$. 

\textbf{Step (3)}.  On the other hand,  for any $x\in X$,  there is $x_{i_0}$ such that $d_X(x, x_{i_0})\leq \epsilon$.

Now we get 
\begin{align}
d_Z(x,  y_{i_0})\leq \epsilon+ d_X(x, x_{i_0})+ d_Y(y_{i_0}, y_{i_0})\leq 2\epsilon.  \nonumber 
\end{align} 
Therefore $X\subseteq \mathbf{U}_{2\epsilon}(Y)\subseteq Z$.

From the above and the definition of Gromov-Hausdorff distance,  the conclusion follows.
}
\qed
	
Now we are ready to prove the main theorem in this section.
\begin{theorem}\label{thm almost rigidity}
	For complete Riemannian manifold $(M^n,g)$ with $K(g)\ge1$ and $\displaystyle \mathcal{IE}((M^n, g),  \mathbb{R}^{n+ 1})\neq \emptyset$,  we have
	\begin{align}
	\frac{d_{GH}((M^n, g), [0, \pi])}{\sqrt{\pi- \mathrm{Diam}_{\mathbb{R}^{n+ 1}}(M^n, g)} }  \leq 4\pi^{\frac{3}{2}}   \nonumber .
	\end{align}
\end{theorem}

\pf
{\textbf{Step (1)}. We firstly choose a map $\mathscr{I}\in \mathcal{IE}((M^n, g),  \mathbb{R}^{n+ 1})$ freely. In the rest argument, we assume $\mathrm{Diam}_{\mathscr{I}}(M^n, g)= |\mathscr{I}(p)- \mathscr{I}(q)|= \pi-\epsilon$ for some $p, q\in M^n$,  where $\epsilon> 0$. Assume $d_{g}(p,q)=\pi-\delta$,  then $\delta\leq \epsilon$.

Without loss of generality, we assume that $\mathscr{I}(q)$ is the origin in $\mathbb{R}^{n+ 1}$, and $\frac{\mathscr{I}(p)- \mathscr{I}(q)}{|\mathscr{I}(p)- \mathscr{I}(q)|}$ is the positive direction of $x_{n+ 1}$-axis.  Define the projection map $P: \mathbb{R}^{n+ 1}\rightarrow\mathbb{R}$, by $P(x_1, \cdots, x_{n+ 1})= x_{n+ 1}$. 


In the rest,  we assume $t\in[0,\pi-\epsilon]$. Define $w_t\vcentcolon= l_{\mathscr{I}(p), \mathscr{I}(q)}\cap P^{-1}(t)$.  For any point $c\in P^{-1}(t)\cap \mathscr{I}(M^n)$, we have $l_{w_t, c}\perp l_{\mathscr{I}(p), \mathscr{I}(q)}$. By Proposition \ref{prop one-side GH-appr}, we know that $|w_t- c|\le \sqrt{\pi\ep}.$ 
	
So $(P^{-1}(t)\cap \mathscr{I}(M^n))\subset \Big(B_{w_t}(\sqrt{\pi\ep})\cap P^{-1}(t)\Big),$ where $B_{w_t}(\sqrt{\pi\ep})$ is the open ball in $\mathbb{R}^{n+ 1}$ centered at $w_t$ with the radius $\sqrt{\pi\ep}$ (see Figure \ref{Cut f(M^n) by P^{-1}(t)}).

\begin{figure}[H]

		\centering
		\includegraphics[width=10cm,height=10cm]{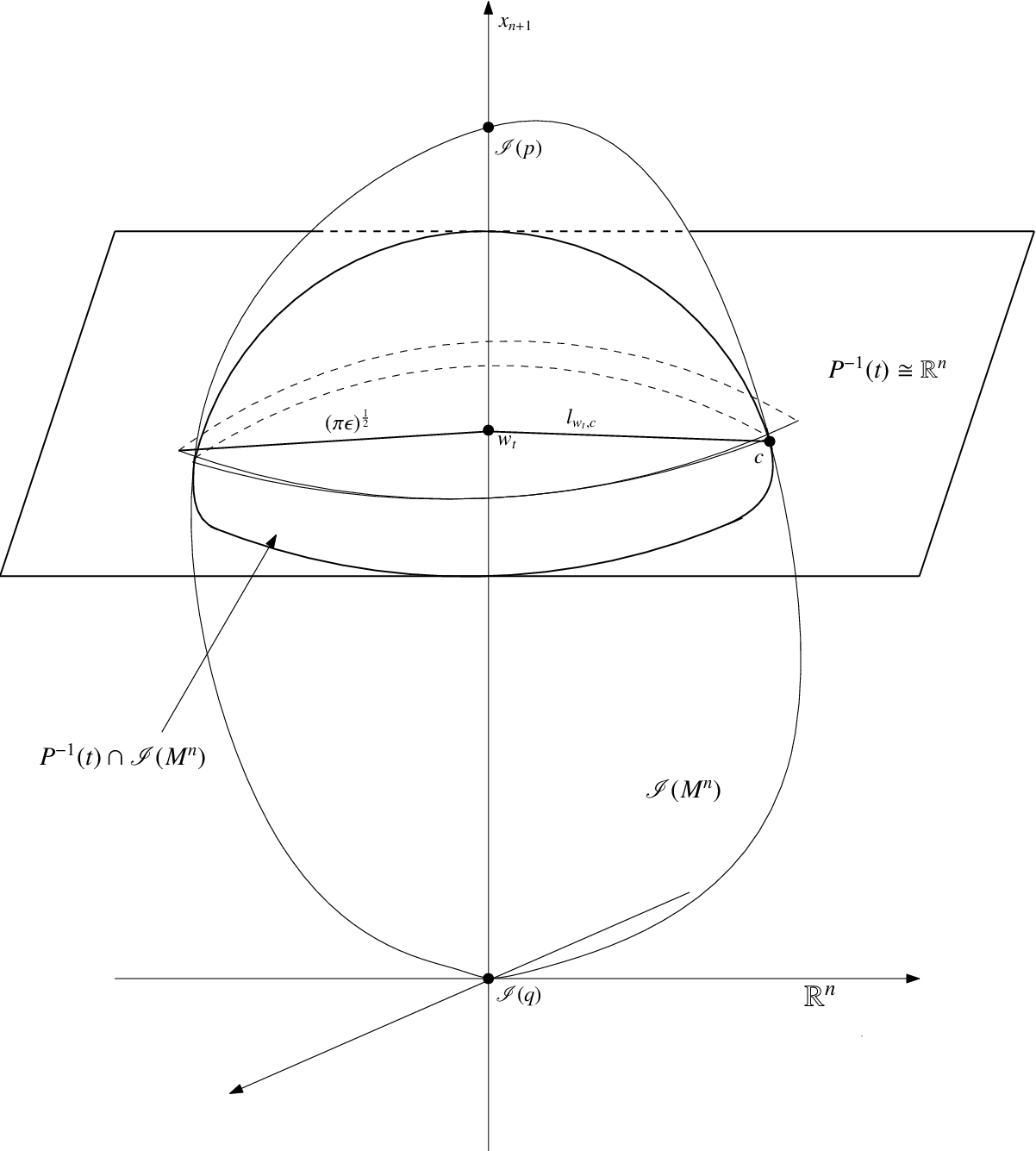}
		\caption{Cut $\mathscr{I}(M^n)$ by $P^{-1}(t)$}
		\label{Cut f(M^n) by P^{-1}(t)}

\end{figure}
}
	
From Lemma \ref{lem from Rc>0 to prin curv > 0},  we get that $\mathscr{I}(M^n)$ is a strictly convex hypersurface in $\mathbb{R}^{n+ 1}$.  For any distinct two points $y_1, y_2\in P^{-1}(t)\cap \mathscr{I}(M^n)$,  consider the $2$-dim plane $\mathbf{P}$ determined by $w_t, y_1, y_2$,  then $\gamma\vcentcolon= \mathbf{P}\cap \mathscr{I}(M^n)$ is a closed convex curve in $\mathbf{P}= \mathbb{R}^2$ (see Figure \ref{Cut f(M^n) by P}).  

\begin{figure}
	\centering
	\includegraphics[width=7cm,height=15cm]{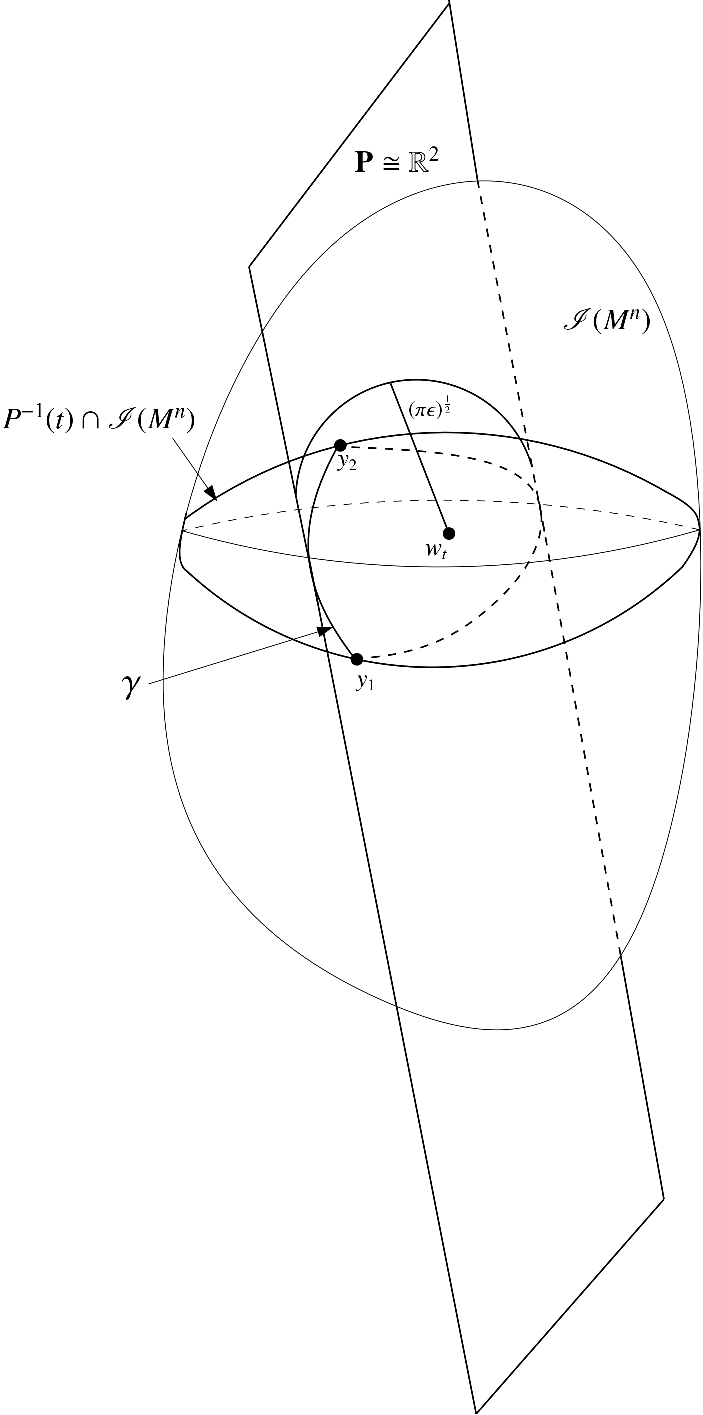}
	\caption{Cut $\mathscr{I}(M^n)$ by $\mathbf{P}$}
	\label{Cut f(M^n) by P}
\end{figure}

From the above,  we get that 
\begin{align}
\gamma \subseteq (\mathbf{P}\cap B_{w_t}(\sqrt{\pi\ep})).  \label{curve is in disk-1} 
\end{align}

By Lemma \ref{lem length comp for convex hypersurface} and (\ref{curve is in disk-1}), we have
\begin{align}
d_g(\mathscr{I}^{-1}(y_1), \mathscr{I}^{-1}(y_2))\leq \frac{1}{2}\ell(\gamma)\le \pi\sqrt{\pi \epsilon}, \quad \quad \quad \forall y_1\neq y_2\in P^{-1}(t)\cap \mathscr{I}(M^n). \nonumber 
\end{align}

Therefore we obtain
\begin{align}
d_g(p_1, p_2)\le \pi\sqrt{\pi \epsilon}, \quad \quad  \forall p_1, p_2\in M^n \ \text{with}\ P(\mathscr{I}(p_1))= P(\mathscr{I}(p_2)). \label{curve length on level sets}
\end{align}

\textbf{Step (2)}. Define $h:M^n\rightarrow\mathbb{R}$, by $h(z)= P(\mathscr{I}(z))$ for any point $z\in M$. Then the range of $h$ is $[0,\pi-\epsilon].$ 

Assume $\gamma_{q,p}$  is one unit speed, geodesic segment from $q$ to $p$ in $(M^n, g)$. Define the map $ G: [0, \pi]\rightarrow (M^n, g)$ as follows:
\begin{equation}\nonumber 
G(t)= \left\{
	\begin{array}{rl}
		&\gamma_{q,p}(t- \frac{\delta}{2}) \ ,  \quad \quad \quad \quad\quad   \quad t\in [\frac{\delta}{2}, \pi- \frac{\delta}{2}], \\
		&q ,   \quad \quad \quad \quad\quad   \quad t\in [0, \frac{\delta}{2}],\\
		&p, \quad \quad \quad \quad\quad   \quad t\in [\pi- \frac{\delta}{2}, \pi].
	\end{array} \right.
	\end{equation}
	
For $t_1, t_2\in [0, \pi]$,  note $\epsilon\leq \pi$, we have 
	\begin{align}
	\sup_{t_1, t_2\in [0, \pi]}\Big||t_1- t_2|- d_g(G(t_1), G(t_2))\Big|\leq \delta\leq \epsilon\leq \sqrt{\pi\epsilon}. \nonumber 
\end{align}	 

For any $y\in M^n$, we define $\varphi: M^n\rightarrow \mathbb{R}^{n+ 1}$ by requiring $\varphi(y)\in (\mathscr{I}(\gamma_{q,p})\cap P^{-1}(h(y))) \subseteq \mathbb{R}^{n+ 1}$ (note the choice of $\varphi$ is possibly not unique).

\begin{figure}
	\centering
	\includegraphics[width=10cm,height=10cm]{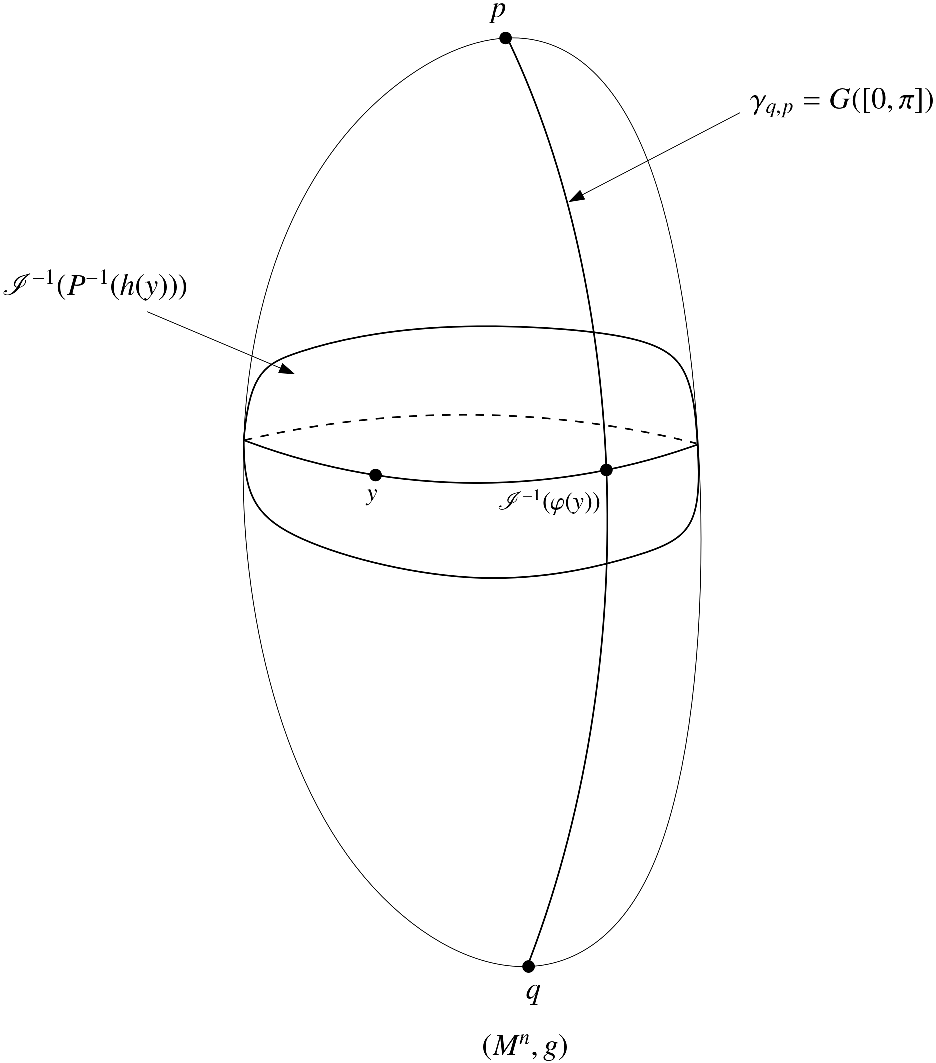}
	\caption{The $(\pi^{\frac{3}{2}}\cdot \sqrt{\epsilon})$-Gromov-Hausdorff approximation }
	\label{gh}
\end{figure}

Note $\mathscr{I}(y), \varphi(y)\in P^{-1}(h(y))$. By (\ref{curve length on level sets}), we obtain
\begin{align}
d_g(y, G[0, \pi])= d_g(y, \gamma_{q,p})\leq d_g(y, \mathscr{I}^{-1}(\varphi(y)))\leq \pi\sqrt{\pi \epsilon}. \nonumber 
\end{align}

Hence $G$ is an $(\pi^{\frac{3}{2}}\cdot \sqrt{\epsilon})$-Gromov-Hausdorff approximation from $[0, \pi]$ to $(M^n, g)$ (See Figure \ref{gh}).

By all the above and Lemma \ref{lem GH map implies GH dist control} , we get 
	\begin{align}
	d_{GH}((M^n, g), [0, \pi]) \leq 4\cdot \pi^{\frac{3}{2}}\cdot \sqrt{\pi- \mathrm{Diam}_{\mathscr{I}}(M^n, g)}\nonumber .
	\end{align}
	Because $\mathscr{I}$ is freely chosen from $\mathcal{IE}((M^n, g),  \mathbb{R}^{n+ 1})$, the conclusion follows. 

\qed

\section*{Acknowledgments}
We thank Tobias Holck Colding for his interest.  The second author is indebted to Jian Ge for helpful discussion during the preparation of this paper, and we also thank his comments and suggestion on the paper.  Last but not least,  we are grateful to Joel Spruck for his comments. 

\textbf{Data availability} Data sharing not applicable to this article as no datasets were generated
or analysed during the current study.

\textbf{Declarations}

\textbf{Conflict of interests} The authors declare that they have no conflict of interest.


\begin{bibdiv}
\begin{biblist}

\bib{AG}{article}{
    AUTHOR = {Abresch, Uwe},
    author = {Gromoll, Detlef},
     TITLE = {On complete manifolds with nonnegative {R}icci curvature},
   JOURNAL = {J. Amer. Math. Soc.},
  FJOURNAL = {Journal of the American Mathematical Society},
    VOLUME = {3},
      YEAR = {1990},
    NUMBER = {2},
     PAGES = {355--374},
      ISSN = {0894-0347,1088-6834},
   MRCLASS = {53C21},
  MRNUMBER = {1030656},
MRREVIEWER = {Ji-Ping\ Sha},
       DOI = {10.2307/1990957},
       URL = {https://doi.org/10.2307/1990957},
}

 

\bib{Brezis}{book}{
    AUTHOR = {Brezis, Haim},
     TITLE = {Functional analysis, {S}obolev spaces and partial differential
              equations},
    SERIES = {Universitext},
 PUBLISHER = {Springer, New York},
      YEAR = {2011},
     PAGES = {xiv+599},
      ISBN = {978-0-387-70913-0},
   MRCLASS = {35-01 (46-01 46E35 46N20 47F05)},
  MRNUMBER = {2759829},
MRREVIEWER = {Vicen\c{t}iu\ D.\ R\u{a}dulescu},
}

\bib{CE}{book}{
    AUTHOR = {Cheeger, Jeff},
    author= {Ebin, David G.},
     TITLE = {Comparison theorems in {R}iemannian geometry},
      NOTE = {Revised reprint of the 1975 original},
 PUBLISHER = {AMS Chelsea Publishing, Providence, RI},
      YEAR = {2008},
     PAGES = {x+168},
      ISBN = {978-0-8218-4417-5},
   MRCLASS = {53C20},
  MRNUMBER = {2394158},
       DOI = {10.1090/chel/365},
       URL = {https://doi.org/10.1090/chel/365},
}

 \bib{Cheng}{article}{
 	AUTHOR = {S.Y. Cheng},
 	TITLE = {Eigenvalue comparison theorem and its geometric applications},
 	JOURNAL = {Math.Z.},
 	VOLUME = {143},
 	YEAR = {1975},
 	PAGES = {289-297}, 	
 }

\bib{Colding-shape}{article}{
    AUTHOR = {Colding, Tobias H.},
     TITLE = {Shape of manifolds with positive {R}icci curvature},
   JOURNAL = {Invent. Math.},
  FJOURNAL = {Inventiones Mathematicae},
    VOLUME = {124},
      YEAR = {1996},
    NUMBER = {1-3},
     PAGES = {175--191},
      ISSN = {0020-9910,1432-1297},
   MRCLASS = {53C23 (53C21)},
  MRNUMBER = {1369414},
MRREVIEWER = {Man\ Chun\ Leung},
       DOI = {10.1007/s002220050049},
       URL = {https://tlink.lib.tsinghua.edu.cn:443/https/443/org/doi/yitlink/10.1007/s002220050049},
}

\bib{Colding-large}{article}{
    AUTHOR = {Colding, Tobias H.},
     TITLE = {Large manifolds with positive {R}icci curvature},
   JOURNAL = {Invent. Math.},
  FJOURNAL = {Inventiones Mathematicae},
    VOLUME = {124},
      YEAR = {1996},
    NUMBER = {1-3},
     PAGES = {193--214},
      ISSN = {0020-9910,1432-1297},
   MRCLASS = {53C23 (53C21)},
  MRNUMBER = {1369415},
MRREVIEWER = {Man\ Chun\ Leung},
       DOI = {10.1007/s002220050050},
       URL = {https://tlink.lib.tsinghua.edu.cn:443/https/443/org/doi/yitlink/10.1007/s002220050050},
}

\bib{EG}{book}{
    AUTHOR = {Evans, Lawrence C.},
    author= {Gariepy, Ronald F.},
     TITLE = {Measure theory and fine properties of functions},
    SERIES = {Textbooks in Mathematics},
   EDITION = {Revised},
 PUBLISHER = {CRC Press, Boca Raton, FL},
      YEAR = {2015},
     PAGES = {xiv+299},
      ISBN = {978-1-4822-4238-6},
   MRCLASS = {28-01},
  MRNUMBER = {3409135},
}

\bib{Gromov-book}{book}{
    AUTHOR = {Gromov, Misha},
     TITLE = {Metric structures for {R}iemannian and non-{R}iemannian
              spaces},
    SERIES = {Progress in Mathematics},
    VOLUME = {152},
      NOTE = {Based on the 1981 French original [ MR0682063 (85e:53051)],
              With appendices by M. Katz, P. Pansu and S. Semmes,
              Translated from the French by Sean Michael Bates},
 PUBLISHER = {Birkh\"auser Boston, Inc., Boston, MA},
      YEAR = {1999},
     PAGES = {xx+585},
      ISBN = {0-8176-3898-9},
   MRCLASS = {53C23 (53-02)},
  MRNUMBER = {1699320 (2000d:53065)},
MRREVIEWER = {Igor Belegradek},
} 

\bib{Nash}{article}{
    AUTHOR = {Nash, John},
     TITLE = {The imbedding problem for {R}iemannian manifolds},
   JOURNAL = {Ann. of Math. (2)},
  FJOURNAL = {Annals of Mathematics. Second Series},
    VOLUME = {63},
      YEAR = {1956},
     PAGES = {20--63},
      ISSN = {0003-486X},
   MRCLASS = {53.1X},
  MRNUMBER = {75639},
MRREVIEWER = {J.\ Schwartz},
       DOI = {10.2307/1969989},
       URL = {https://tlink.lib.tsinghua.edu.cn:443/https/443/org/doi/yitlink/10.2307/1969989},
}

\bib{Nash-C1}{article}{
    AUTHOR = {Nash, John},
     TITLE = {{$C^1$} isometric imbeddings},
   JOURNAL = {Ann. of Math. (2)},
  FJOURNAL = {Annals of Mathematics. Second Series},
    VOLUME = {60},
      YEAR = {1954},
     PAGES = {383--396},
      ISSN = {0003-486X},
   MRCLASS = {53.0X},
  MRNUMBER = {65993},
MRREVIEWER = {S.\ Chern},
       DOI = {10.2307/1969840},
       URL = {https://tlink.lib.tsinghua.edu.cn:443/https/443/org/doi/yitlink/10.2307/1969840},
}

	\bib{PP}	{book}{
	title={Riemannian geometry},
	author={Petersen, Peter},
	volume={171},
	year={2006},
	publisher={Springer}
}		

\bib{Spruck}{article}{
    AUTHOR = {Spruck, Joel},
     TITLE = {On the radius of the smallest ball containing a compact
              manifold of positive curvature},
   JOURNAL = {J. Differential Geometry},
  FJOURNAL = {Journal of Differential Geometry},
    VOLUME = {8},
      YEAR = {1973},
     PAGES = {257--258},
      ISSN = {0022-040X,1945-743X},
   MRCLASS = {53C20},
  MRNUMBER = {339005},
MRREVIEWER = {J.\ E.\ Brothers},
       URL = {http://tlink.lib.tsinghua.edu.cn:80/http/80/org/projecteuclid/yitlink/euclid.jdg/1214431642},
}

\bib{VH}{article}{
    AUTHOR = {Van Heijenoort, John},
     TITLE = {On locally convex manifolds},
   JOURNAL = {Comm. Pure Appl. Math.},
  FJOURNAL = {Communications on Pure and Applied Mathematics},
    VOLUME = {5},
      YEAR = {1952},
     PAGES = {223--242},
      ISSN = {0010-3640,1097-0312},
   MRCLASS = {52.0X},
  MRNUMBER = {52131},
MRREVIEWER = {H.\ Busemann},
       DOI = {10.1002/cpa.3160050302},
       URL = {https://tlink.lib.tsinghua.edu.cn:443/https/443/org/doi/yitlink/10.1002/cpa.3160050302},
} 

\bib{WZ}{article}{
    AUTHOR = {Wang, Bing},
    author= {Zhao, Xinrui},
     TITLE = {Canonical diffeomorphisms of manifolds near spheres},
   JOURNAL = {J. Geom. Anal.},
  FJOURNAL = {Journal of Geometric Analysis},
    VOLUME = {33},
      YEAR = {2023},
    NUMBER = {9},
     PAGES = {Paper No. 304, 31},
      ISSN = {1050-6926,1559-002X},
   MRCLASS = {53C20},
  MRNUMBER = {4615511},
       DOI = {10.1007/s12220-023-01375-x},
       URL = {https://tlink.lib.tsinghua.edu.cn:443/https/443/org/doi/yitlink/10.1007/s12220-023-01375-x},
}

\end{biblist}
\end{bibdiv}
\end{document}